\documentclass[12pt]{article}

\usepackage{times}

\usepackage[a4paper,text={128mm,185mm}, centering]{geometry}

\usepackage{changepage}

\usepackage{titlesec}

\titleformat{\section}[hang]%
{\bfseries\large}{\thesection.}{1ex}{}%

\titleformat{\subsection}[hang]%
{\bfseries}{\thesubsection}{1ex}{}%

\usepackage{amsthm}
\usepackage{fancyhdr}
\theoremstyle{plain}
\newtheorem{theorem}{Theorem}[section]
\newtheorem{lemma}[theorem]{Lemma}
\newtheorem{proposition}[theorem]{Proposition}
\newtheorem{corollary}[theorem]{Corollary}
\newtheorem{definition}[theorem]{Definition} 

\theoremstyle{definition}

\newtheorem{remark}[theorem]{Remark}

\usepackage[matrix,arrow,curve,cmtip]{xy}
\usepackage{amssymb}
\usepackage{latexsym}
\usepackage{graphicx}

\title{\vskip 5pt \bf Level~$\epsilon$}
\author{\itshape\bfseries {Francisco MARMOLEJO \quad Mat\'\i as MENNI}}
\date{}

\begin{document}
\maketitle


\newcommand{\yy}{\ensuremath{\mathbf{y}}}

\newcommand{\Set}{\mathbf{Set}}
\newcommand{\con}{\mathbf{Set}}
\newcommand{\Sets}{\ensuremath{\Set}}

\newcommand{\Top}{\mathbf{Top}}
\newcommand{\Arc}{\mathbf{Arc}}
\newcommand{\Doc}{\mathbf{Doc}}
\newcommand{\Lin}{\mathbf{Lin}}
\newcommand{\LH}{\mathbf{LH}}

\newcommand{\cat}[1]{\mathcal{#1}} 
\newcommand{\escat}[1]{\cat{#1}}
\newcommand{\psh}[1]{\widehat{#1}}
\newcommand{\Psh}[1]{\widehat{#1}}
\newcommand{\Sh}{\mathrm{Sh}}
\newcommand{\Sep}{\mathrm{Sep}}
\newcommand{\colim}{colim}
\newcommand{\dom}{dom}

\newcommand{\calA}{\ensuremath{\escat{A}}} 
\newcommand{\calB}{\ensuremath{\escat{B}}} 
\newcommand{\calC}{\ensuremath{\escat{C}}} 
\newcommand{\calD}{\ensuremath{\escat{D}}} 
\newcommand{\calE}{\ensuremath{\escat{E}}} 
\newcommand{\calF}{\ensuremath{\escat{F}}} 
\newcommand{\calH}{\ensuremath{\escat{H}}} 
\newcommand{\calI}{\ensuremath{\escat{I}}} 
\newcommand{\calK}{\ensuremath{\escat{K}}} 
\newcommand{\calL}{\ensuremath{\escat{L}}} 
\newcommand{\calQ}{\ensuremath{\escat{Q}}} 
\newcommand{\calS}{\ensuremath{\escat{S}}} 
\newcommand{\calT}{\ensuremath{\escat{T}}} 
\newcommand{\calW}{\ensuremath{\escat{W}}} 
\newcommand{\calX}{\ensuremath{\escat{X}}} 
\newcommand{\calZ}{\ensuremath{\escat{Z}}} 

\newcommand{\bfA}{\mathbf{A}}
\newcommand{\bfL}{\mathbf{L}}

\newcommand{\opCat}[1]{\ensuremath{{#1}^{\mathrm{op}}}}

\newcommand{\Nat}{\ensuremath{\mathbb{N}}}
\newcommand{\Int}{\ensuremath{\mathbb{Z}}}
\newcommand{\reals}{\mathbb{R}}
\newcommand{\complex}{\mathbb{C}}


\newcommand{\Rig}{\mathbf{Rig}}
\newcommand{\Ring}{\mathbf{Ring}}
\newcommand{\riRig}{\mathbf{riRig}}
\newcommand{\FinSet}{\Set_f}


\newcommand{\Cat}{\mathbf{Cat}}
\newcommand{\rGrph}{\psh{\Delta_1}}

\newcommand{\Ext}{\mathbf{Ext}}

\newcommand{\Fam}{\mathrm{Fam}}
\newcommand{\Hom}{\mathrm{Hom}}

\newcommand{\oalpha}{\overline{\alpha}} 
\newcommand{\obeta}{\overline{\beta}} 
\newcommand{\oeta}{\overline{\eta}} 
\newcommand{\osigma}{\overline{\sigma}}
\newcommand{\otau}{\overline{\tau}} 
\newcommand{\otheta}{\overline{\theta}}
\newcommand{\okappa}{\overline{\kappa}}
\newcommand{\ophi}{\overline{\phi}} 


\newcommand{\Nil}{\textrm{Nil}}
\newcommand{\Red}{\mathbf{Red}}

\newcommand{\twopl}[2]{\ensuremath{\left\langle #1, #2 \right\rangle}}

\newcommand{\OmegaBase}{\mathit{2}}


\cfoot{}
\thispagestyle{empty}
\vskip 25pt
\begin{adjustwidth}{0.5cm}{0.5cm}
{\small
{\bf R\'esum\'e.} 
Lawvere a observ\'e que certains `gros' topos en g\'eom\'etrie alg\'ebrique sugg\`erent l'existence d'un `niveau infinit\'esimal', \'etroitement li\'e aux alg\`ebres locales de dimension finie. Motiv\'es par cette observation, nous proposons une d\'efinition \'el\'ementaire de {\em level~$\epsilon$} associ\'ee \`a un morphisme géométrique local, établissons quelques propriétés de base pertinentes suggérées par l'intuition géométrique et donnons une description concrète du niveau~$\epsilon$ déterminé par plusieurs morphismes géométriques pré-cohésifs. \\
{\bf Abstract.} 
Lawvere has observed that certain  `gros' toposes in algebraic geometry suggest the existence of an `infinitesimal level', closely related to finite-dimensional local algebras. Motivated by this observation we propose an elementary definition of  {\em level~$\epsilon$} associated to a  local geometric morphism, establish some relevant basic properties suggested by geometric intuition, and give concrete descriptions of the level~$\epsilon$ determined by several pre-cohesive geometric morphisms.
}
\end{adjustwidth}


\section{Introduction}

The formulation of Axiomatic Cohesion \cite{Lawvere07} and its development in the last ten years naturally invites to revisit the ideas and concrete problems outlined in \cite{Lawvere91}. We consider here a specific question in the dimension theory proposed in Section~{II} of the latter reference:

\begin{quotation}
The infinitesimal spaces, which contain the base topos in its non-Becoming aspect, are a crucial step toward determinate Becoming, but fall short of having among themselves enough connected objects, i.e. they do not in themselves constitute fully a `category of cohesive unifying Being.' In examples the four adjoint functors relating their topos to the base topos coalesce into two (by the theorem that a finite-dimensional local algebra has a unique section of its residue field) and 
the infinitesimal spaces may well negate
the largest essential subtopos of the ambient one which has that property. 
This level may be called `dimension~$\epsilon$'; calling the levels (i.e. the subtoposes essential over the base) `dimensions' does not imply that they are linearly ordered nor that the Aufhebung process touches each of them. The infinitesimal spaces provide (in many
ways) a good example of a non trivial unity-and-identity-of-opposites inside the ambient topos of Being: explicitly recognizing the {\em two} inclusions, as spaces which could be called infinitesimal and formal spaces respectively, may help clarify the confusing but powerful
interplay between these two classes which are opposite but in themselves
identical. The calculation of the $\epsilon$-skeleton and $\epsilon$-coskeleton,
of a space which is neither, needs to be carried out, and also the calculation
of the Aufhebung of dimension~$\epsilon$.
\end{quotation}

Our purpose is to confirm the suggestion that, in many examples of cohesion, there exists  a ``largest essential subtopos of the ambient one which has" the property that ``the four adjoint functors to the base coalesce into two".
In fact, we turn the suggestion into a rigorous definition of the level~$\epsilon$ determined, if it exists, by a local geometric morphism ${\calE \rightarrow \calS}$.
When it exists, it is an essential subtopos ${\calE_{\epsilon} \rightarrow \calE}$ with special properties.  We prove that level~$\epsilon$  exists in many examples and  we give an explicit description. In particular, if $\calE$ is the Zariski topos determined by the field of complex numbers, the site for $\calE_{\epsilon}$ will be shown to be closely related to local algebras, as suggested by the quotation above.

Although some of the theory is developed in more generality, the typical topos of Being that we have in mind is the domain of a pre-cohesive geometric morphism as defined, for example, in \cite{LawvereMenni2015}. We recall most of the definitions but the reader is assumed to be familiar with the ideas therein.

In Section~\ref{SecLevelsAndDimensions} we recall in more detail the basics of the dimension theory mentioned above. One way to start is to fix a local geometric morphisms ${p : \calE \rightarrow \calS}$ (typically with extra properties), consider its centre ${\calS \rightarrow \calE}$ as `level 0' and study the levels above it.
In Section~\ref{SecSubtoposesAboveCentreNew} we analyse the subtoposes above the centre of a local presheaf topos.
The level~$\epsilon$ determined by a local geometric morphism is defined in Section~\ref{SecEpsilon}. In the remaining sections we calculate the level~$\epsilon$ of several examples. In Section~\ref{SecPresheaf} we  analyse the level~$\epsilon$ of local presheaf toposes in general and in some simple cases. 
In Section~\ref{SecWeilSubquality} we show that the Weil topos (determined by the field $\complex$ of complex numbers) underlies a  subquality of  the Gaeta topos determined by the same field.
This  is shown to be level~$\epsilon$  in  Section~\ref{SecGaeta}.

\begin{remark} Since Gaeta toposes are perhaps not yet widely known, we include here a brief description. If $\calD$ is a small extensive category then the finite families ${(D_i \rightarrow D \mid i\in I)}$ of maps in $\calD$ such that the induced ${\sum_{i\in I} D_i \rightarrow D}$ is an isomorphism form the basis of a Grothendieck topology.
The associated topos ${G\calD}$ of sheaves is called the {\em Gaeta topos} (of $\calD$) and it is equivalent to the category of finite-product preserving functors ${\opCat{\calD} \rightarrow \Set}$. The `Gaeta topology' is subcanonical so that the Yoneda embedding of $\calD$ into the topos of presheaves factors through the Gaeta topos but, moreover, the factorization ${\calD \rightarrow G\calD}$ preserves finite coproducts. See, for example, the end of page~3 in \cite{Lawvere91} or  Section~2 in \cite{Lawvere08}.
If $T$ is an algebraic theory whose category $\calA$ of finitely presented algebras is coextensive then  we may naturally refer to ${G(\opCat{\calA})}$ as the Gaeta topos determined by $T$. For instance, we have the Gaeta toposes determined rigs, by distributive lattices or by $k$-algebras, where $k$ is a ring. We may even push the terminology further and simply speak (as we have done above) of the Gaeta topos determined by $\complex$, instead of the Gaeta topos determined by the theory of $\complex$-algebras.
\end{remark}

 In Section~\ref{SecZariski} we show that the Zariski topos determined by $\complex$ (which is not a presheaf topos) also has a level~$\epsilon$ and that it coincides with that of the Gaeta topos.

\section{Levels and dimensions}
\label{SecLevelsAndDimensions}

   In this section we recall some of the material in Section~{II} of \cite{Lawvere91} which proposes to consider essential subtoposes of a given topos of spaces (or a {\em category of Being}) as a refined notion of `dimensions' in that topos. The quotations in this section are taken from that reference.
First notice that reflective subcategories of a fixed category $\calX$ may be partially ordered as follows.

\begin{lemma}\label{generalsobreadjuntos}
If the adjunctions ${F\dashv U:\calA\to\calX}$ and ${L\dashv R:\calC\to\calX}$ are  such that ${U:\calA\to\calX}$ and ${R:\calC\to \calX}$ are full and faithful then the following conditions are equivalent:
\begin{itemize}
\item[\textup{(i)}] There is an adjunction $H\dashv K:\calA\to\calC$ with $K$ full and faithful such that $RK\simeq U$
(or, equivalently, $HL\simeq F$).
\item[\textup{(ii)}]  $U:\calA\to \calX$ factors (up to iso) through $R:\calC\to\calX$.
\item[\textup{(iii)}]  $F:\calX\to\calA$ factors (up to iso) through $L:\calX\to\calC$.
\item[\textup{(iv)}] ${\nu_U:U\to RLU}$ is invertible, where ${\nu : 1_{\calX} \rightarrow R L}$  is the unit of ${L \dashv R}$.
\item[\textup{(v)}] $F\nu:F\to FRL$ is invertible.  
\end{itemize}
\end{lemma}

\begin{proof} Clearly (i) implies (iii), and (iii) trivially implies (v) since $R$ is full and faithful. 

Assume (v) and let ${\xi : L R \rightarrow 1_{\calC}}$ be the counit of ${L \dashv R}$.
Define $H=FR$ and $K=LU$; we show that $H\dashv K$ by showing that the 2-cells
\[
\xy
(0,9)*+{\calC}="1";
(10,0)*+{\calX}="2",
(20,-9)*+{\calA}="3",
(30,0)*+{\calX}="4",
(40,9)*+{\calC}="5",
\POS "1" \ar_{R} "2",
\POS "2" \ar^(.4){1_\calX} "4",
\POS "1" \ar^{1_\calC} "5",
\POS "2" \ar_{F} "3",
\POS "3" \ar_{U} "4",
\POS "4" \ar_{L} "5",
\POS (26,8) \ar@{=>}_{\xi^{-1}} (26,2),
\POS (21,-1) \ar@{=>}_{\eta} (21,-6),
\endxy
\textup{ \ \ \ \  }
\xy
(0,-9)*+{\calA}="1";
(10,0)*+{\calX}="2",
(20,9)*+{\calC}="3",
(30,0)*+{\calX}="4",
(40,-9)*+{\calA.}="5",
\POS "1" \ar^{U} "2",
\POS "2" \ar_(.3){F} "5",
\POS "1" \ar_{1_\calA} "5",
\POS "2" \ar^{L} "3",
\POS "3" \ar^{R} "4",
\POS "4" \ar^{F} "5",
\POS (24,4) \ar@{=>}_{(F\nu)^{-1}} (24,-1),
\POS (12,-3) \ar@{=>}_{\varepsilon} (12,-8),
\endxy
\]
satisfy the triangular identities, where $\eta$ and $\varepsilon$ are the unit and counit of $F\dashv U$. One of these
triangular identities is trivial. The other is
\[
\xy
(0,-9)*+{\calA}="1";
(10,0)*+{\calX}="2",
(20,9)*+{\calC}="3",
(30,0)*+{\calX}="4",
(40,-9)*+{\calA}="5",
(50,0)*+{\calX}="6",
(60,9)*+{\calC}="7",
\POS "1" \ar^{U} "2",
\POS "2" \ar_(.3){F} "5",
\POS "1" \ar_{1_\calA} "5",
\POS "2" \ar^{L} "3",
\POS "3" \ar^{R} "4",
\POS "3" \ar^{1_\calC} "7",
\POS "4" \ar^{F} "5",
\POS "4" \ar^{1_\calX} "6",
\POS "5" \ar_{U} "6",
\POS "6" \ar_{L} "7",
\POS (24,4) \ar@{=>}_{(F\nu)^{-1}} (24,-1),
\POS (12,-3) \ar@{=>}_{\varepsilon} (12,-8),
\POS (50,8) \ar@{=>}_{\xi^{-1}} (50,3),
\POS (43,-1) \ar@{=>}_{\eta\ } (43,-5),
\endxy
= 1_{LU}.
\]
This equation is equivalent to 
\[
\xy
(10,-9)*+{\calA}="1";
(15,0)*+{\calX}="2",
(20,9)*+{\calC}="3",
(30,0)*+{\calX}="4",
(40,-9)*+{\calA}="5",
(50,0)*+{\calX}="6",
(60,9)*+{\calC}="7",
\POS "1" \ar^{U} "2",
\POS "2" \ar^{L} "3",
\POS "3" \ar_{R} "4",
\POS "3" \ar^{1_\calC} "7",
\POS "4" \ar_{F} "5",
\POS "4" \ar^{1_\calX} "6",
\POS "5" \ar_{U} "6",
\POS "6" \ar_{L} "7",
\POS (50,8) \ar@{=>}_{\xi^{-1}} (50,3),
\POS (43,-1) \ar@{=>}_{\eta\ } (43,-5),
\endxy
=
\xy
(20,9)*+{\calA}="3",
(30,0)*+{\calX}="4",
(40,-9)*+{\calC}="5",
(50,0)*+{\calX}="6",
(60,9)*+{\calA}="7",
(65,0)*+{\calX}="8",
(70,-9)*+{\calC.}="9",
\POS "3" \ar_{U} "4",
\POS "3" \ar^{1_\calA} "7",
\POS "4" \ar_{L} "5",
\POS "4" \ar^{1_\calX} "6",
\POS "5" \ar_{R} "6",
\POS "6" \ar_{F} "7",
\POS "7" \ar^{U} "8",
\POS "8" \ar^{L} "9",
\POS (50,8) \ar@{=>}_{\varepsilon^{-1}} (50,3),
\POS (43,-1) \ar@{=>}_{\nu} (43,-5),
\endxy
\]
To see that this equation holds, replace on the right $U\varepsilon^{-1}$ by $\eta_U$, and then replace 
$L\nu$ by $\xi^{-1}_L$. Observe furthermore, that the counit of $H\dashv K$ is invertible, thus $K$ is
fully faithful. We conclude that (v) implies (i).

The proof that (i) $\Rightarrow$ (ii) $\Rightarrow$ (iv) $\Rightarrow$ (i) is very similar.
\end{proof}

If the equivalent conditions of Lemma~\ref{generalsobreadjuntos} hold then we may say that
the reflective subcategory ${L \dashv R}$ is {\em above} ${F \dashv U}$.

\begin{remark}\label{remarka}
In the situation of Lemma \ref{generalsobreadjuntos}, we may as well assume (as we do in what follows) 
that the adjunction $H\dashv K:\calA\to\calC$ is given by $FR\dashv LU$ with unit and counit given by 
$$\xymatrix{
1_\calC \ar[r]^-{\xi^{-1}} & L R \ar[rr]^-{L\eta_R} && L U F R  \textnormal{ \ \ and \ \ } 
  F R L U \ar[rr]^-{(F\nu_U)^{-1}} && F U \ar[r]^-{\varepsilon} & 1_\calA
}$$
respectively. Observe as well that, if $F$ preserves finite limits, then $H$ also preserves them; and
if $F$ has a left adjoint, then $H$ also has a left adjoint.
\end{remark}

From now on we restrict attention to the case where the ambient category $\calX$ is a fixed topos $\calE$.
In this case Lemma~\ref{generalsobreadjuntos} and Remark~\ref{remarka} imply  the following.

\begin{corollary}\label{estaeralaproposicion3.1}
Given subtoposes $j:\calE_j\to\calE$ and $k:\calE_k\to\calE$ the following conditions are equivalent.
\begin{itemize}
\item[\textup{(i)}] $\calE_k$ is above $\calE_j$. 
\item[\textup{(ii)}] $j_*:\calE_j\to\calE$ factors through $k_*:\calE_k\to\calE$.
\item[\textup{(iii)}] $j^*:\calE\to\calE_j$ factors through $k^*:\calE\to\calE_k$.
\item[\textup{(iv)}] The natural transformation ${\nu_{j_*}: j_* \rightarrow k_* k^* j_*}$ is an isomorphism.
\item[\textup{(v)}] The natural transformation ${j^* \nu : j^*  \rightarrow j^* k_* k^*}$ is an isomorphism. 
\end{itemize}
When this is the case, we can take as witness of the fact that $\calE_k$ is above $\calE_j$ the
geometric morphism  $h:\calE_j\to\calE_k$ such that $h^*=j^*k_*$, $h_*=k^*j_*$ with unit and counit
given by
$$\xymatrix{
1_{k} \ar[r]^-{{\xi'}^{-1}} & k^* k_* \ar[rr]^-{k^*\nu_{k_*}} && k^* j_* j^* k_* 
 \textnormal{ \  and  \ }
j^* k_* k^* j_* \ar[rr]^-{(j^*\nu'_{j_*})^{-1}} && j^* j_* \ar[r]^-{\xi} & 1_{j}
}$$
respectively, where $\nu,\xi$ are the unit and counit of $j:\calE_j\to\calE$ and $\nu',\xi'$ are the corresponding
ones for $k:\calE_k\to\calE$.
\end{corollary}

A geometric morphism ${f : \calF \rightarrow \calE}$ is called {\em essential} if the inverse image $f^*$ has a left adjoint ${f_! : \calF \rightarrow \calE}$.
Following \cite{Lawvere91}, essential subtoposes of $\calE$  will be called {\em levels}. Notice that for any given level ${l : \calE_l \rightarrow \calE}$, the leftmost adjoint ${l_!}$ is full and faithful (because the direct image $l_*$ is).
Levels may be partially ordered according to their underlying subtoposes as in Corollary~\ref{estaeralaproposicion3.1}.

\begin{quotation}
 The basic idea is simply to identify dimensions
with levels and then try to determine what the general dimensions
are in particular examples. More precisely, a space may be
said to have (less than or equal to) the dimension grasped by a given
level if it belongs to the negative (left adjoint inclusion) incarnation
of that level. 
\end{quotation}

So, for any level ${l : \calE_l \rightarrow \calE}$ and any $X$ in $\calE$,  the counit ${l_! (l^* X) \rightarrow X}$ may be called the {\em $l$-skeleton} of $X$.
The object $X$ is said to be {\em $l$-skeletal} if its $l$-skeleton is an iso, so that
${l_! : \calE_l \rightarrow \calE}$ is the full subcategory of $l$-skeletal objects.
On the other hand, and in accordance with standard terminology, the objects in the full  subcategory ${l_* : \calE_l \rightarrow \calE}$ will be called {\em $l$-sheaves}.

A subtopos ${\calE_j \rightarrow \calE}$  is {\em way-above} a level  ${l : \calE_l \rightarrow \calE}$ if $j$ is above $l$ and, moreover, ${l_! : \calE_l \rightarrow \calE}$ factors through ${j_* : \calE_j \rightarrow \calE}$.
The {\em Aufhebung} of level $l$ is (when it exists) the smallest level of $\calE$ that is way-above $l$.

The Aufhebung of a level need not be easy to calculate. For an illustration of the complexity of the issue see  \cite{Lawvere1989taco}, \cite{KellyLawvere89}, \cite{Roy97}, \cite{Lawvere2002a} and \cite{KennetEtAl2011}.

Recall that a geometric morphism ${p : \calE \rightarrow \calS}$ is {\em local} if ${p_* : \calE \rightarrow \calS}$ has a fully faithful right adjoint (usually denoted by ${p^!}$). For such a $p$, the subtopos ${p_* \dashv p^! : \calS \rightarrow \calE}$ is a level  called the {\em centre} of $p$ and it is sometimes convenient to think of it as the smallest non-trivial level of $\calE$ (or `dimension~0'), especially, if $p$ has further properties:

\begin{quotation}
Within the class of all levels over the base (of course it is a
set in fact if the category of Being is a topos), the base itself is often
further distinguished by having a still further left adjoint to its
discrete inclusion, this extra functor therefore assigning to every
space in Being its set of components.
\end{quotation}

So let us fix a local and essential geometric morphism ${p : \calE \rightarrow \calS}$.
Recall that {\em essential} means that the fully faithful  ${p^* : \calS \rightarrow \calE}$  has a further  left adjoint ${p_! : \calE \rightarrow \calS}$. As quoted above, this left adjoint is thought of as assigning, to each space (i.e. an object in $\calE$), its associated set (i.e. object in $\calS$) of pieces or connected components.
To aid the intuitive discussion, the centre of $p$ will be called  {\em level 0} (of $p$) and its Aufhebung will be called  {\em level~1} (of $p$).

%
%
%
%

\begin{quotation}
Because of the special feature of dimension
zero of having a components functor to it (usually there is
no analogue of that functor in higher dimensions), the definition of
dimension one is equivalent to the quite plausible condition: the
smallest dimension such that the set of components of an arbitrary
space is the same as the set of of components of the skeleton at that
dimension of the space, or more pictorially: if two points of any
space can be connected by anything, then they can be connected by
a curve. Here of course by ``curve" we mean any figure in (i.e. map
to) the given space whose domain is one-dimensional.
\end{quotation}

See Proposition~17 in \cite{Lawvere1989taco} and Proposition in p.~{19} of \cite{Roy97}.
We give a different proof:

\begin{proposition}\label{PropLawvereRoy} For any level   ${l}$   above level~0, $l$  is way-above 0  if and only if, for every $X$ in $\calE$,
${p_! ( l_! ( l^* X)) \rightarrow p_! X}$ is an isomorphism  (where ${ l_! ( l^* X) \rightarrow X}$ is the $l$-skeleton of $X$).
\end{proposition}
\begin{proof}
Apply Lemma~\ref{generalsobreadjuntos} to  ${l^* \dashv l_* : \calE_l \rightarrow \calE}$ and ${p_! \dashv p^* : \calS \rightarrow \calE}$.
\end{proof}

Notice that this result does not assume that level~1 exists. The levels way-above level~0 may be considered as the levels above level~1 even if the latter does no exist. 

\begin{corollary}\label{CorLawvereRoy} Assume that level~1 exists. If a level ${l}$ is above level $0$ then, $l$ is above 1 if and only if ${p_! ( l_! ( l^* X)) \rightarrow p_! X}$ is an isomorphism for every $X$ in $\calE$.
\end{corollary}

We also recall a general description of levels in  presheaf toposes.

\begin{definition}\label{DefIdempotentIdeal}{\em
An {\em ideal}  of a small category $\calC$ is a class of maps $\calI$ in $\calC$ that satisfies the following two conditions:%
\begin{enumerate}
\item (Right ideal) For every ${g : D \rightarrow C}$ in $\calI$ and ${h : E\rightarrow D}$ in  $\calC$, ${gh \in \calI}$.
\item (Left ideal) For every ${f : C \rightarrow B}$ in $\calC$ and ${g : D\rightarrow C}$ in $\calI$, ${fg \in \calI}$.
\end{enumerate}
An ideal is called {\em idempotent} if for every ${f \in \calI}$ there are ${g, h\in \calI}$  such that ${f = g h}$.}
\end{definition}

Theorem~{4.4} in \cite{KellyLawvere89} shows that levels of the presheaf topos $\Psh{\calC}$ are in bijective correspondence with idempotent ideals of $\calC$. 
If $\calI$ is such an idempotent ideal then the associated Grothendieck topology $J$ is such that,  for each $C$ in $\calC$, a sieve $S$ on $C$ is $J$-covering if and only if it contains all the maps in $\calI$ with codomain $C$.

\section{Subtoposes above the centre of a local map}
\label{SecSubtoposesAboveCentreNew}


Let ${p : \calE \rightarrow \calS}$ be a local geometric morphism.

\begin{proposition}\label{PropAboveCentre}
Let $j:\calE_j\to\calE$ be a subtopos and assume that the following diagram
\[
\xy
(0,6)*+{\calE_j}="1",
(12,6)*+{\calE}="2",
(12,-6)*+{\calS}="3",
\POS "1" \ar^j "2",
\POS "1" \ar_f "3",
\POS "2" \ar^p "3",
\endxy
\]
commutes so that ${p_* j_* = f_* : \calE_j \to \calS}$ and ${j^* p^* = f^* : \calS\to\calE_j}$. Then the following are equivalent,
where $\nu$ is the unit of ${j : \calE_j \to \calE}$:
\begin{enumerate}
\item The subtopos $j:\calE_j\to\calE$ is above the centre of $p$.
\item $p^!:\calS\to\calE$ factors through $j_*:\calE_j\to\calE$.
\item $p_*:\calE\to\calS$ factors through $j^*:\calE\to\calE_j$.
\item The natural transformation $\nu_{p^!}:p^!\to j_*j^*p^!$ is an isomorphism. 
\item The natural transformation $p_*\nu:p_*\to p_*j_*j^*$ is an isomorphism.
\item The geometric morphism $f$ is local and $j_*f^!\simeq p^!$.
\end{enumerate}
In this case, we may assume that
$f^!=j^*p^!$ and that the unit $\overline{\eta}$ and counit $\overline{\varepsilon}$ of $f_*\dashv f^!$ are given by
$$\xymatrix{
1_{\calE_j} \ar[r]^-{\xi^{-1}} & j^* j_* \ar[rr]^-{j^* \eta_{j_*}} &&  j^* p^! p_* j_* = f^! f_* \\
f_* f^! = p_* j_* j^* p^! \ar[rr]^-{(p_*\nu_{p^!})^{-1}} && p_* p^! \ar[r]^-{\varepsilon} & 1_\calS
}$$
respectively, where $\eta$ and $\varepsilon$ are the unit and counit of ${p_* \dashv p^!}$ respectively, and $\xi$ is the
(iso) counit of $j^*\dashv j_*$.
\end{proposition}
\begin{proof} Corollary \ref{estaeralaproposicion3.1} tells us that items 1 to 5 are equivalent; furthermore, Remark \ref{remarka} tells us that we can take unit and counit
of $f_*=p_*j_*\dashv j^*p^!$ as given in the statement of the proposition. So that $f$ is local.
Observe that $j_*f^!=j_*j^*p^!\simeq p^!$ via the iso $\nu_{p^!}$. So any one of the first five conditions implies 6.
Finally, almost immediately,  6 implies~2. 
\end{proof}

So subtoposes above the centre determine local maps towards the base.
Moreover, if $p$ is essential then so is $f$. Indeed, ${f_! = p_! j_*}$ and the composites
$$\xymatrix{
1_{\calE_j} \ar[r]^-{\xi^{-1}} & j^* j_* \ar[rr]^-{j^* \sigma_{j_*}} && j^* p^* p_! j_* = f^* f_! \\
 f_! f^* = p_! j_* j^* p^* \ar[rr]^-{(p_! \nu_{p^*})^{-1}} && p_! p^* \ar[r]^-{\tau}  & 1_{\calS}
}$$
are the unit and counit of ${f_! \dashv f^*}$, where $\sigma$ and $\tau$ are the unit and counit of ${p_! \dashv p^*}$.

For instance, consider a small category $\calC$ with a terminal object so that the canonical geometric morphism ${\Psh{\calC} \rightarrow \Set}$ is local. See C3.6.3(b) in \cite{elephant}.

\begin{lemma}\label{LemForCorLargestEssSubQ} Let ${\calD \rightarrow \calC}$ be a full subcategory such that idempotents split in $\calD$.
Then, the induced essential subtopos ${\Psh{\calD} \rightarrow  \Psh{\calC}}$ is above the centre of $p$ if and only if the subcategory ${\calD \rightarrow \calC}$ contains the terminal object.
\end{lemma}
\begin{proof}
If the subcategory ${\calD \rightarrow \calC}$ contains the terminal object then ${\Psh{\calD} \rightarrow \Set}$ is local by C3.6.3(b) in \cite{elephant} and it is straightforward to check that there is a natural iso as in item~6 of Proposition~\ref{PropAboveCentre}, 
so ${\Psh{\calD} \rightarrow  \Psh{\calC}}$ is above the centre of $p$.

Conversely, assume that ${\Psh{\calD} \rightarrow  \Psh{\calC}}$ is above the centre of $p$.
Then ${\Psh{\calD} \rightarrow \Set}$ is local by item~6 of Proposition~\ref{PropAboveCentre} so, as idempotents split by hypothesis, 
$\calD$ must have a terminal object $1_{\calD}$ by C3.6.3(b) in \cite{elephant}.
So it remains to show that the inclusion ${\calD \rightarrow \calC}$ preserves the terminal object.
To do this let $1_{\calC}$ be the terminal object of $\calC$ and notice that, by   item~3 of  Proposition~\ref{PropAboveCentre}, it must be the case that, for every $X$ in ${\Psh{\calC}}$, ${X 1_\calC \cong X 1_{\calD}}$.
 Taking ${X = \calC(\_, 1_{\calD})}$ we may conclude that ${1_{\calD}}$ has a point and, as ${\calD \rightarrow \calC}$ is fully faithful,
it must be the case that the composite ${1_{\calD} \rightarrow 1_{\calC} \rightarrow 1_{\calD}}$ is the identity on $1_{\calD}$.
So the point ${1_{\calC} \rightarrow 1_{\calD}}$ is an iso.
\end{proof}

On the other hand,  it is not the case that subtoposes that induce local maps are above the centre in general.
For example,  let ${\calC = \{ 0 < \frac{1}{2} < 1 \} }$ be the total order with three elements
and consider the full subcategory ${\calD = \{0 < \frac{1}{2} \} \rightarrow \calC}$ then the following diagram
$$\xymatrix{
\Psh{\calD} \ar[rd] \ar[r] & \Psh{\calC} \ar[d] \\
& \Set
}$$
commutes and both morphisms to $\Set$ are local, but the subtopos ${\Psh{\calD} \rightarrow \Psh{\calC}}$ is not above the center of $p$; as one may show, for example, by checking that ${p_* : \Psh{\calC} \rightarrow \Set}$ does not invert the unit of the subtopos.

Assume from now on that every object of $\calC$ has a point so that the canonical 
${ p : \Psh{\calC} \rightarrow \Set}$ is pre-cohesive.
Let $J$ be a Grothendieck topology on $\calC$ and let ${\Sh(\calC, J) \rightarrow \Psh{\calC}}$ be the associated subtopos.

\begin{lemma}\label{LemGrothendieckAboveCodiscrete}
The following are equivalent:
\begin{enumerate}
\item The subtopos ${\Sh(\calC, J) \rightarrow \Psh{\calC}}$ is above the centre of $p$.
\item For every $C$ in $\calC$ and ${S \in J C}$, $S$ contains all points of $C$.
\item The maximal sieve is the only $J$-cover of $1$. 
\end{enumerate}
\end{lemma}
\begin{proof}
By Corollary~{4.5} in \cite{LawvereMenni2015}, the centre of ${p : \Psh{\calC} \rightarrow \Set}$ coincides with the subtopos of sheaves for the double negation topology.
In other words, ${p^! : \Set \rightarrow \Psh{\calC}}$ coincides with the subtopos  ${\Sh(\calC, K) \rightarrow \Psh{\calC}}$ where a sieve on $C$ is $K$-covering if and only if it contains all points of $C$.
So ${\Sh(\calC, J) \rightarrow \Psh{\calC}}$ is above the centre of $p$ if and only if, for every $C$ in $\calC$,  ${J C \subseteq K C}$.
In other words, the first two items are equivalent.
The second item trivially implies the third.
The third item easily implies the second.
\end{proof}

\section{Subqualities and level $\epsilon$}
\label{SecEpsilon}

Recall that if we denote the counit of ${p_* \dashv p^!}$ by  ${\varepsilon}$, and the unit and counit of ${p^* \dashv p_*}$ by $\alpha$ and $\beta$ then the following diagram commutes
$$\xymatrix{
  p^*  \ar[d]_-{p^* \varepsilon^{-1}} \ar[r]^-{\eta_{p^*}} & p^! p_* p^* \ar[d]^-{p^! \alpha^{-1}} \\
  p^* p_* p^! \ar[r]_-{\beta_{p^!}} & p^!
}$$
and  the composite is denoted by ${\phi : p^* \rightarrow p^!}$. 
Following item~{(c)} in Definition~2 of \cite{Lawvere07} we could say that the {\em Nullstellensatz} holds (for $p$) if ${\phi : p^* \rightarrow p^!}$ is monic.
Recall that the Nullstellensatz holds  if and only if $p$ is hyperconnected.
See  \cite{Johnstone2011} for explicit proofs of the equivalences between different formulations of the Nullstellensatz.

\begin{lemma}\label{LemQ1} If the subtopos ${j : \calE_j \rightarrow \calE}$ is above the centre of ${p :  \calE \rightarrow \calS}$ and ${f = p j : \calE_j \rightarrow \calS}$ denotes the composite local geometric morphism  then, ${\ophi : f^* \rightarrow f^!}$ and ${j^* \phi : j^* p^* \rightarrow j^* p^!}$ are equal. Therefore, if $p$ is hyperconnected then so is $f$.
\end{lemma}
\begin{proof}
Observe that in the commutative diagram below
$$\xymatrix{
j^* p^* \ar[rrd]_-{j^* \phi} \ar[rr]^-{j^* p^* \varepsilon^{-1}} && j^* p^* p_* p^! \ar[d]^-{j^* \beta_{p^!}} \ar[rr]^-{j^* p^* p_* \nu_{p^!}} && j^* p^* p_* j_* j^* p^! \ar[d]^-{j^* \beta_{j_* j^* p^!}} \\
    && j^* p^! \ar[rrd]_-{id} \ar[rr]^-{j^* \nu_{p^!}} && j^* j_* j^* p^! \ar[d]^-{\xi_{j^* p^!}} \\
 && && j^* p^!
}$$
the top composite equals ${f^* \overline{\varepsilon}^{-1}}$ by Proposition~\ref{PropAboveCentre}, while the composite on the right, as in Section~{6.1} of \cite{MarmolejoMenni2017}, is ${\overline{\beta}_{f^!}}$, so the top-right composite is ${\overline{\phi} : f^* \rightarrow f^!}$.
\end{proof}

It is convenient to slightly extend the terminology in \cite{Lawvere07} and say that a local geometric morphism ${q : \calQ \rightarrow \calS}$ is  a {\em quality type} if the canonical transformation ${\phi : q^* \rightarrow q^!}$ is an isomorphism. 
Such special adjunctions are also called {\em quintessential localizations} in \cite{Johnstone96}. Notice that, trivially,  quality types  satisfy the Nullstellensatz, so they are hyperconnected.

Fix a local geometric morphism  ${p : \calE \rightarrow \calS}$ and call its centre level~0 as in Section~\ref{SecLevelsAndDimensions}.

\begin{definition}\label{DefSubquality}{\em A {\em subquality} of ${p : \calE \rightarrow \calS}$ is a subtopos ${j : \calE_j \rightarrow \calE}$  above level~0  and such that the composite ${p j : \calE_j \rightarrow \calS}$ is a quality type.}
\end{definition}

Compare with the notion of {\em quality} introduced in \cite{Lawvere07}.
Roughly speaking, while quality is a (special kind of) functor  to a quality type,  a subquality is a (special kind of) functor from a quality type.

\begin{lemma}\label{LemForPropThmCharSubQ} Let ${\calE_j \rightarrow \calE}$ be a subtopos above level~0.
Then, ${\calE_j \rightarrow \calE}$ is a subquality  of $p$ if and only if ${j^* \phi : j^* p^* \rightarrow j^* p^!}$ is an isomorphism.
Hence, if $p$ satisfies the Nullstellensatz then, ${\calE_j \rightarrow \calE}$ is  a subquality  of $p$ if and only if ${\phi_A : p^* A \rightarrow p^! A}$ is $j$-dense for every $A$ in $\calS$.
\end{lemma}
\begin{proof}
Follows from Lemma~\ref{LemQ1}.
\end{proof}

A subquality of $p$ is said to be {\em essential} if it is so as a subtopos of $\calE$. 
In this case, the subquality is a level of $\calE$ and it is above the centre of  $\calE$.

\begin{definition}\label{DefLevelEpsilon} {\em
In case it exists, {\em level~$\epsilon$} (of $p$) is the largest essential subquality of $p$.}
\end{definition}

Intuitively,  level~$\epsilon$ is an `infinitesimal' dimension so it should not be  above level~1. 
More generally,  essential subqualities  should not be  way-above level $0$.
In the context  of Proposition~\ref{PropLawvereRoy} we can make this precise as follows.
Let $\Omega$ be the subobject classifier of $\calE$ and recall (Proposition~3 in  \cite{Lawvere07}) that if $p$ is a quality type and ${p_! \Omega = 1}$ then $\calS$ is degenerate. Intuitively, the condition ${p_! \Omega = 1}$  is a positive way of saying that $p$ is not a quality type.

\begin{proposition}\label{PropWayAboveZero} Let the local  ${p : \calE \rightarrow \calS}$ be essential and hyperconnected. Let  ${l : \calE_l \rightarrow \calE}$ be an essential subquality of $p$.  If ${p_! \Omega = 1}$  and $l$ is  way-above $0$ then $\calS$ is degenerate.
\end{proposition}
\begin{proof}
Let ${\rho : l_! (l^* \Omega) \rightarrow \Omega}$ be the $l$-skeleton of $\Omega$.
As $l$ is way-above level~0 by hypothesis, Proposition~\ref{PropLawvereRoy} implies that ${p_! \rho : p_! (l_! (l^* \Omega)) \rightarrow p_! \Omega = 1}$ is an isomorphism. 
As  ${p l : \calE_l \rightarrow \calS}$ is a quality type, ${p_! (l_! (l^* \Omega)) \cong p_* (l_* (l^* \Omega))}$.
So ${ p_* (l_* (l^* \Omega)) = 1}$.

 Let ${\Omega_{l}}$ be the subobject classifier in $\calE_l$.
 It is well-known that ${l_* \Omega_{l}}$ is a retract of $\Omega$ in $\calE$.
 So ${\Omega_{l} \cong l^* (l_* \Omega_{l})}$ is a retract of ${l^* \Omega}$,
and then ${p_* (l_*  \Omega_{l})}$ is a retract of ${p_* (l_* (l^* \Omega)) = 1}$.
That is, ${p_* (l_*  \Omega_{l}) = 1}$.
As ${p l : \calE_l \rightarrow \calS}$  is hyperconnected,  ${p_* l_* : \calE_l \rightarrow \calS}$ preserves the subobject classifier. Altogether, the subobject classifier of $\calS$ is terminal.
\end{proof}

\begin{corollary} Let ${p : \calE \rightarrow \calS}$ be essential and hyperconnected. Assume that ${p_! \Omega = 1}$ and that  level~$\epsilon$ of $p$ exists. If  $\epsilon$ is way-above  $0$ then $\calE$ is degenerate. 
\end{corollary}

\section{Level~$\epsilon$ in presheaf toposes}
\label{SecPresheaf}

Consider a small category $\calC$ with terminal object so that the canonical geometric morphism ${p : \Psh{\calC}\rightarrow \Set}$ is local.
Without loss of generality we may assume that idempotents split in $\calC$.

\begin{corollary}\label{CorLargestEssentialSubQTNew}
If ${\calD \rightarrow \calC}$ is a full subcategory closed under splitting of idempotents then, the essential subtopos ${\Psh{\calD} \rightarrow \Psh{\calC}}$ is a subquality of  $p$ if and only if ${\calD \rightarrow \calC}$ contains the terminal object and every object of $\calD$ has a unique point.
\end{corollary}
\begin{proof}
It follows from Lemma~\ref{LemForCorLargestEssSubQ} above and Proposition~{4.5} in \cite{Menni2014a} which implies that the restriction ${\Psh{\calD} \rightarrow \Set}$ is a quality type if and only if every object in $\calD$ has a unique point.
\end{proof}

Let ${\calC_{!} \rightarrow \calC}$  be the full subcategory of all objects in $\calC$ that have exactly one point. For later reference we emphasize the following consequence of 
Corollary~\ref{CorLargestEssentialSubQTNew}:

\begin{lemma}\label{LemShrieck} The subcategory ${\calC_{!} \rightarrow \calC}$ is the largest full subcategory ${\calD\rightarrow \calC}$ of $\calC$ such that ${\Psh{\calD} \rightarrow \Psh{\calC}}$ is an essential subquality.
\end{lemma}

We  discuss below some related sufficient conditions  for this subquality to be level~$\epsilon$. In order to do so recall (C2.2.18 in \cite{elephant}) that an object $B$ in a site  ${(\calB, J)}$  is {\em $J$-irreducible} if every $J$-covering sieve on $B$ is the maximal sieve.
The Grothendieck coverage $J$ is said to be {\em rigid} if, for every object $B$ in $\calB$, the family of all morphisms from $J$-irreducible objects to $B$ generates a $J$-covering sieve. If $J$ is rigid and ${\calI \rightarrow \calB}$ is the full subcategory of $J$-irreducible objects then the Comparison Lemma implies that restriction along the inclusion ${\calI \rightarrow \calB}$ restricts to an equivalence ${\Sh(\calB, J) \cong \Psh{\calI}}$.

\begin{proposition}\label{PropRigid} If every Grothendieck coverage on $\calC$ is rigid  then ${\Psh{\calC_{!}} \rightarrow \Psh{\calC}}$ is level~$\epsilon$ of the local ${\Psh{\calC} \rightarrow \Set}$.
\end{proposition}
\begin{proof}
If every Grothendieck coverage on $\calC$ is rigid  then the levels of  ${\Psh{\calC}}$ are all induced  by full subcategories of $\calC$. So the result follows from Lemma~\ref{LemShrieck}.
\end{proof}

At first glance, Proposition~\ref{PropRigid} may look difficult to apply so let us derive a simpler sufficient condition.

\begin{corollary}\label{CorFinite} If $\calC$ is finite then ${\Psh{\calC_{!}} \rightarrow \Psh{\calC}}$ is level~$\epsilon$ of the local ${\Psh{\calC} \rightarrow \Set}$.
\end{corollary} 
\begin{proof}
If $\calC$ is finite then every coverage of $\calC$ is rigid (see C2.2.21 in \cite{elephant}).
\end{proof}

In particular, graphic toposes \cite{Lawvere1989taco} have a level~$\epsilon$ of this simple kind.
On the other hand, it is worth mentioning that Proposition~\ref{PropRigid} also applies to non-finite examples such as the sites studied in \cite{KennetEtAl2011}. For instance,   ${\Delta}$ or the category of non-empty finite sets. It follows that   simplicial sets and the classifier of non-trivial Boolean algebras have a level~$\epsilon$. In this cases, though, level~$\epsilon$ coincides with level~0 because, in the respective sites, the terminal object is the only object with exactly one point.

In order to discuss a simple example where $\epsilon$ does not coincide with $0$ we borrow the 4-element graphic monoid  discussed in p.~62  of \cite{Lawvere1989taco}.  Consider first, as an auxiliary step, the pre-cohesive topos ${ p : \Psh{\Delta_1} \rightarrow \Set}$ of reflexive graphs. 
Let $G$ be the graph with two nodes and a non-trivial loop displayed below
$$\xymatrix{
\ar@(ld, lu)[] \bot &  \top 
}$$
and let  $M$ be the monoid of endomorphisms of $G$ that are either constant or don't collapse the non-trivial loop.
There are four such maps, two constants, the identity, and the unique map $\alpha$ that sends $\top$ to $\bot$ but does not collapse the loop. 
If we split the constants, we obtain the (non-full)  subcategory of ${\Psh{\Delta_1}}$ pictured below
$$\xymatrix{
1 \ar[r]<-1ex>_-{{\bot}} \ar[r]<+1ex>^-{{\top}} & G \ar@(ru,rd)^-{\alpha}
}$$
where $1$ is terminal, ${\alpha\alpha = \alpha}$ and ${\alpha \bot = \bot = \alpha {\top}}$. (Notice that, we are not drawing constant endos or the unique map to the terminal.) It is then clear that we may describe an object of ${\Psh{M}}$ as a reflexive graph equipped with an idempotent function on its edges that sends each edge $x$ to a loop ${x \cdot \alpha}$ on the  domain of $x$, preserving the identity loops. As suggested in \cite{Lawvere1989taco} we call ${x\cdot \alpha}$ the {\em preparation to do $x$}.
Alternatively, as a graph equipped with a distinguished subset of loops containing the trivial ones, and a domain-preserving retraction for the inclusion of distinguished loops into edges.

To calculate level~$\epsilon$ of the pre-cohesive ${\Psh{M}}$ we split all idempotents. Let ${s : D \rightarrow G}$ be the split monic that results from splitting $\alpha$ in $\Psh{\Delta_1}$ and let ${r : G \rightarrow D}$ be its retraction. We may picture the idempotent-splitting $N$ of $M$ as the (non-full)  subcategory of ${\Psh{\Delta_1}}$  suggested below
$$\xymatrix{
 & D \ar[d]<-1ex>_-s \\
1 \ar@(u,l)[ru]^-{\ddagger} \ar[r]<-1ex>_-{{\bot}} \ar[r]<+1ex>^-{{\top}} & G \ar@(ru,rd)^-{\alpha} \ar[u]<-1ex>_-r 
}$$
with ${r \bot = \ddagger = r\top : 1 \rightarrow D}$ and ${s \ddagger = \bot : 1 \rightarrow G}$.

It is then clear that the full subcategory ${N_{!} \rightarrow N}$ is that determined by $D$ and $1$ and, by Corollary~\ref{CorFinite}, 
${\Psh{N_{!}} \rightarrow \Psh{N} \cong \Psh{M}}$ is level~$\epsilon$ of the pre-cohesive ${\Psh{M} \rightarrow \Set}$.
 It is then possible to check that the $\epsilon$-skeletal objects in $\Psh{M}$ are those that consist only of distinguished loops. On the other hand, the sheaves for level~$\epsilon$ are the objects  such that for each distinguished loop $d$ and each node $n$ there exists a unique edge to $n$ with preparation $d$.

The topos ${\Psh{M}}$ does not have many levels so it is easy to see that the Aufhebung of level~$\epsilon$ coincides  with the top level, that is, the whole of ${\Psh{M}}$. Similarly, level~1 must also be the top level in this case.

Consider again a small category $\calC$ with a terminal object and such that every object has a point, so that the canonical  ${p : \Psh{\calC} \rightarrow \Set}$ is pre-cohesive.

\begin{definition}\label{DefPseudoConstant}{\em
A morphism ${f:D\to C}$ in $\calC$ is a {\em pseudo-constant} if for any two points ${a, b: 1\to D}$ in $\calC$, ${f a = f  b : 1 \rightarrow C}$. }
\end{definition}

In other words, the pseudo-constants are those morphisms that are constant on points. We  think of a pseudo-constant as a morphism that factors through an object that has exactly one point. Notice that if $D$ has exactly one point then every map ${D \rightarrow C}$ is a pseudo-constant.

\begin{proposition}\label{PropSubQualitiesOfPresheaf}  If $J$ is a Grothendieck topology on $\calC$  such that the subtopos ${\Sh(\calC, J) \rightarrow \Psh{\calC}}$ is above the centre of $p$ then the following are equivalent:
\begin{enumerate}
\item The subtopos ${\Sh(\calC, J) \rightarrow \Psh{\calC}}$ is a subquality of $p$.
\item For every $C$ in $\calC$, the sieve of all the pseudo-constants with codomain $C$ is $J$-covering.
\item For every $C$ in $\calC$, ${J C}$ contains a sieve of pseudo-constants.
\end{enumerate}
\end{proposition}
\begin{proof}
First consider the canonical ${\phi : p^* \rightarrow p^! }$ in the present context. For $A$ in $\Set$ and $C$ in $\calC$, the function  ${\phi_{A, C} : A = (p^* A) C \rightarrow (p^! A) C = A^{\calC(1, C)}}$ sends ${a \in A}$ to the constant function in ${A^{\calC(1, C)}}$ that collapses everything to $a$. 

If the first item holds then, for every $A$ in $\Sets$, ${\phi_A : p^* A \rightarrow p^! A}$ is $J$-dense by Lemma~\ref{LemForPropThmCharSubQ}. So $\phi_A$ must be locally surjective (w.r.t. $J$) by Corollary~{III.7.6} in \cite{maclane2}. 
In particular, ${\phi_{\calC(1, C)}}$ must be so.
Take the identity $id$ in codomain of ${\phi_{\calC(1, C), C} : \calC(1, C) \rightarrow  \calC(1, C)^{\calC(1, C)}}$. 
Local surjectivity implies the existence of  a $J$-cover ${(f_i : C_i \rightarrow C \mid i \in I)}$ such that for every ${i \in I}$, ${id \cdot f_i = f_i (\_) \in \calC(1, C)^{\calC(1, C_i)}}$ is constant. In other words, each $f_i$ is a pseudo-constant.
So the third item holds. 
The third item trivially implies the second. 
If the second item holds then we can use the sieve mentioned there to prove that
${\phi_A : p^* A \rightarrow p^! A}$ is locally surjective.
\end{proof}

Notice  that  pseudo-constants in $\calC$ form and ideal in the sense of Definition~\ref{DefIdempotentIdeal}.

\begin{proposition}\label{PropEpsilonExists} If pseudo-constants in $\calC$ form an idempotent ideal then the pre-cohesive ${p : \Psh{\calC} \rightarrow \Set}$ has a level $\epsilon$ and it coincides with the largest subquality of $p$. 
\end{proposition}
\begin{proof}
The Grothendieck topology $J$ on $\calC$ determined by the idempotent ideal of pseudo-constants is such that a sieve on $C$ is $J$-covering if and only if it contains all the pseudo-constants with codomain $C$. It follows that the terminal object is only covered by the identity so ${\Sh(\calC, J) \rightarrow \Psh{\calC}}$ is above the centre of $p$ by Lemma~\ref{LemGrothendieckAboveCodiscrete}. Proposition~\ref{PropSubQualitiesOfPresheaf} implies that the essential subtopos ${\Sh(\calC, J) \rightarrow \Psh{\calC}}$ is a subquality of $p$ and that  every topology $J'$ inducing a subquality of $p$ must satisfy ${J \subseteq J'}$. 
\end{proof}

In the case of reflexive graphs, simplicial sets, or the Gaeta topos determined by the theory of distributive lattices, the site satisfies that every pseudo-constant factors through a point  so, in these cases,  level $\epsilon$ exists and coincides with the centre.

\begin{definition}\label{DefEnoughLittleFigures}{\em 
We say that ${\calC}$ has {\em enough little figures} if for every pseudo-constant ${D \rightarrow C}$  there is a commutative diagram
$$\xymatrix{
D \ar[rd] \ar[r] & B \ar[d] \\
                    & C
}$$
such that ${B}$ has exactly one point. }
\end{definition}

The intuition behind the terminology is that a map ${B \rightarrow C}$ whose domain has exactly one point is to be thought of as a `little figure' of $C$, or a figure of $C$ with `little' domain.

\begin{corollary}\label{CorLargestSubqualityIsEssential} If $\calC$ has enough little figures then ${\Psh{\calC_!} \rightarrow \Psh{\calC}}$ is level~$\epsilon$ of the pre-cohesive ${p : \Psh{\calC} \rightarrow \Set}$ and it coincides with the largest subquality of $p$.
\end{corollary}
\begin{proof}
An object  $C$ in $\calC$ has exactly one point if and only if the identity on $C$ is a pseudo-constant. So a little figure (which is of course a pseudo-constant) factors trivially as a composite of pseudo-constants. Therefore, if $\calC$ has enough little figures then the ideal of pseudo-constants is idempotent. Proposition~\ref{PropEpsilonExists} implies that level $\epsilon$ exists and that it coincides with the largest subquality of $p$. It remains to show that level $\epsilon$ coincides with the indicated presheaf subtopos, but notice that the Grothendieck topology determined by the ideal of pseudo-constants  is rigid  because a sieve on $C$ is covering if and only if it contains all the little figures of $\calC$; that is all the morphisms whose domain has exactly one point. The irreducible objects w.r.t. to this topology are exactly those in ${\calC_!}$ so level $\epsilon$ coincides with ${\Psh{\calC_!} \rightarrow \Psh{\calC}}$.
\end{proof}

\section{The Weil subquality of  the Gaeta topos of $\complex$}
\label{SecWeilSubquality}

All algebras we consider are commutative and unital as in \cite{AtiyahMacdonald}. The following is a straightforward generalization of Definition~{2.14} in \cite{Dubuc1979} allowing an arbitrary base field instead of $\reals$. Let $k$ be a field.

\begin{definition}\label{DefWeilAlgebra} {\em
A {\em Weil algebra (over $k$)} is a $k$-algebra $A$ such that:
\begin{enumerate}
\item $A$ is local, say, with unique maximal ideal $\mathfrak{m}$.
\item The composite ${k \rightarrow A \rightarrow A/\mathfrak{m}}$ is an isomorphism.
\item $A$ is a finite $k$-algebra (i.e. it is finitely generated as a $k$-module).
\item ${\mathfrak{m}^n = 0}$ for some $n$.
\end{enumerate}}
\end{definition}

It is known that there is some redundancy in this definition. 
Compare with the definition of {\em alg\`ebre local} in \cite{Weil1953},
or the definition  in I.16 of \cite{KockSDG2ed}.
What we need to relate Weil algebras with the material in the present paper is the following, surely folk, result.

\begin{lemma}\label{CharComplexWeilAlgebras} For any local $\complex$-algebra $A$ the following are equivalent:
\begin{enumerate}
\item $A$ is a Weil algebra over $\complex$.
\item $A$ is finitely generated.
\item $A$ is Artinian.
\end{enumerate}
\end{lemma}
\begin{proof}
The first item implies the second because, as $A$ is a finite $k$-algebra by hypothesis, then it is finitely generated. Indeed, any basis for the finite dimensional vector space $A$ generates $A$ as a $k$-algebra.

If $A$ is finitely generated then it is a  Jacobson ring by Exercises~{5.23} and~{5.24} in \cite{AtiyahMacdonald}, so every prime ideal is an intersection of maximal ideals.
As $A$ is local, it has a unique prime ideal (which must coincide with the maximal one). In this case, the algebra is Artinian by Exercise~{8.2} op.~cit.

Finally, if $A$ is Artinian and  ${\mathfrak{m}}$ is the unique maximal ideal of $A$ then the composite ${\complex \rightarrow A \rightarrow A/\mathfrak{m}}$ must be an iso by the `weak' version of Hilbert's Nullstellensatz (Corollary~{7.10} op.~cit.). 
Also, $A$ is a finite $k$-algebra by Exercise~{8.3} op.~cit.
Moreover, $\mathfrak{m}$ must be the nilradical of $A$, so $\mathfrak{m}$ is nilpotent by Proposition~{8.4} op.~cit.
\end{proof}

Let $\Ring$ be the category of rings and ${\complex/\Ring}$ be the coslice category of $\complex$-algebras. The full subcategory of finitely generated $\complex$-algebras will be denoted by ${(\complex/\Ring)_{f.g.} \rightarrow \complex/\Ring}$.
The category ${\calD = \opCat{((\complex/\Ring)_{f.g.})}}$ is essentially small and extensive.
The associated Gaeta topos will be denoted by ${\mathfrak{G} = \mathfrak{G}(\calD)}$ and call it the {\em Gaeta topos of $\complex$}.

   The Gaeta topos of $\complex$ is well-known to be a presheaf topos. To recall that description define a ring to be {\em (directly) indecomposable} if it has exactly two idempotents and let ${(\complex/\Ring)_{f.g.i.} \rightarrow (\complex/\Ring)_{f.g.}}$ be the full subcategory of those finitely generated algebras that are indecomposable.
Let ${\calC = \opCat{((\complex/\Ring)_{f.g.i.})}}$ so that the obvious inclusion ${\calC \rightarrow \calD}$ is the subcategory of those objects in $\calD$ that are `connected' in the sense that they have  no non-trivial coproduct decompositions.
   
   The Gaeta topos $\mathfrak{G}$ may be identified with the topos ${\Psh{\calC}}$ of presheaves on $\calC$.
   By Hilbert's Nullstellensatz, every object in $\calC$  has a point so the canonical geometric morphism ${p : \mathfrak{G} \rightarrow \Sets}$ is pre-cohesive. Moreover, there are certainly objects in ${\calC}$ that have more than one point so $p$ is Sufficiently Cohesive.
   
   If we let ${\calW \rightarrow (\complex/\Ring)_{f.g.i.}}$ be the full subcategory of Weil algebras then  ${\mathfrak{W} = \Set^{\calW}}$ is the {\em Weil topos} discussed in \cite{Dubuc1979}. 
   
\begin{proposition}\label{PropWeilSubquality} 
The Weil topos ${\mathfrak{W}}$ is an essential subquality of the Gaeta topos ${\mathfrak{G}}$.
\end{proposition}
\begin{proof}
By Lemma~\ref{CharComplexWeilAlgebras}, the full subcategory ${\opCat{\calC_!} \rightarrow \opCat{\calC} = (\complex/\Ring)_{f.g.i.}}$ coincides with ${\calW \rightarrow (\complex/\Ring)_{f.g.i.}}$. So  ${\mathfrak{W} = \Sets^{\calW} = \Psh{\calC_!} \rightarrow \Psh{\calC} = \mathfrak{G}}$ is an essential subquality by Lemma~\ref{LemShrieck}.
\end{proof}

%
%

\section{The Weil subquality is level~$\epsilon$ of the Gaeta topos}
\label{SecGaeta}

Let ${\calD}$ be a category with terminal object and let ${L : \calD_{\bullet} \rightarrow \calD}$ be the full subcategory determined by the objects whose points are jointly epic.

\begin{lemma}\label{LemEpicPseudoConstantsWithCoReducedDonain} If ${L A}$ has a point and ${e : L A \rightarrow V}$ in $\calD$ is an epic pseudo-constant then ${V = 1}$.
\end{lemma}
\begin{proof}
As $e$ is epic and the points of ${L A}$ are jointly epic, the family of all composites 
$$\xymatrix{
1 \ar[r] & L A  \ar[r]^-e & V
}$$
is jointly epic but, as $e$ is pseudo-constant, there is only one such map, so we have an epic ${1 \rightarrow V}$ which of course is also split monic, so $V$ is terminal.
\end{proof}

Natural further hypotheses allow us to deal with more pseudo-constants.

\begin{lemma}\label{LemEpicPseudoConstants} Assume that ${L : \calD_{\bullet} \rightarrow \calD}$ has an epic-preserving right adjoint with monic counit.
If ${X}$ has a point then, for every epic pseudo-constant ${f : X \rightarrow Y}$, $Y$ has exactly one point.
\end{lemma}
\begin{proof}
Let $R$ be the right adjoint to $L$ and denote the counit by ${\beta}$.
As ${L (R 1) = 1}$,  ${L (R X)}$ must have a point because $X$ does by hypothesis. 
Also, the map ${R f}$ is epic by hypotheses and then so is ${L (R f) : L (R X) \rightarrow L (R Y)}$. 
Moreover, it is a pseudo-constant because,  ${\beta_Y : L (R Y) \rightarrow Y}$ is monic and ${\beta_Y (L (R f)) = f \beta_X}$.
 Lemma~\ref{LemEpicPseudoConstantsWithCoReducedDonain} implies that ${L (R Y) = 1}$.
As every point of $Y$ factors through ${ \beta : L(R Y) \rightarrow Y}$, $Y$ has exactly one point.
\end{proof}

The next result supplies many little figures.

\begin{proposition}\label{PropAbstractEnoughPseudoConstants} If every map in $\calD$ factors as an epi followed by a mono  and ${L : \calD_{\bullet} \rightarrow \calD}$ has an epi-preserving right adjoint with monic counit then  every pseudo-constant whose domain has a point factors via an object with exactly one point.
\end{proposition}
\begin{proof}
Let ${f}$ be  pseudo-constant whose domain has a point. 
By hypothesis,  ${f =  m e }$ for some monic $m$ and epic $e$.  Then the codomain of $e$ has exactly one point by Lemma~\ref{LemEpicPseudoConstants}.
\end{proof}

The proof of Proposition~\ref{PropAbstractEnoughPseudoConstants} is, in essence, that  in \cite{Cornulier2018}. This will become evident below where we discuss the context of Cornulier's Mathoverflow answer.

Let $\Ring$ be the category of (commutative unital) rings and consider the full subcategory  ${\Red \rightarrow \Ring}$  of reduced rings (i.e. those whose only nilpotent element is $0$).

\begin{lemma}\label{LemReduced}
This inclusion ${\Red \rightarrow \Ring}$ has  a left adjoint that preserves monomorphisms.
Moreover, the unit of the adjunction is regular epic.
\end{lemma}
\begin{proof}
The left adjoint sends $R$ in $\Ring$ to ${R/\Nil(R)}$ where ${\Nil(R)}$ is the nilradical of $R$.
See Proposition~{1.7} in \cite{AtiyahMacdonald}. The unit ${R \rightarrow R/\Nil(R)}$ is a regular epimorphism and the left adjoint ${\Ring \rightarrow\Red}$ preserves monos because if ${m : R \rightarrow S}$ is a monomorphism then ${m^* \Nil(S)  = \Nil(R)}$ as subsets of $R$. 
\end{proof}

 Let $\mathbb{C}$ be the field of complex numbers and consider the coslice category ${\mathbb{C}/\Ring}$ of {\em $\mathbb{C}$-algebras}.

\begin{lemma}\label{LemJacobson} A finitely generated $\complex$-algebra $R$ is reduced (as a ring) if and only if the family of all maps ${R \rightarrow \complex}$ is jointly monic.
\end{lemma}
\begin{proof}
If the family  of maps ${R \rightarrow \complex}$ is jointly monic then $R$ is, as a ring, a subobject of a power of $\complex$ so it is reduced.
Conversely, assume that $R$ is reduced. That is,  the nilradical ${\Nil(R)}$ is trivial. For finitely generated algebras over a field, the nilradical equals the Jacobson radical (see Exercise~{5.24} in \cite{AtiyahMacdonald}), so the intersection of the maximal ideals in $R$ is $0$. In other words, the collection of all maps ${R \rightarrow \complex}$ is jointly monic.
\end{proof}

Let ${(\complex/\Ring)_{f.g.}}$ be the category of finitely generated $\complex$-algebras.

\begin{lemma}\label{LemPropIsApplicable} 
Every pseudo-constant in ${\opCat{((\complex/\Ring)_{f.g.})}}$ whose domain has a point factors via an object with exactly one point.
\end{lemma}
\begin{proof}
It is enough to check that ${\calD = \opCat{((\complex/\Ring)_{f.g.})}}$ satisfies the hypotheses of Proposition~\ref{PropAbstractEnoughPseudoConstants}. 
It is well-known that $\calD$ has epi/regular-mono factorizations so it remains to show that the inclusion ${\calD_\bullet \rightarrow \calD}$ has an epi-preserving right adjoint 
with monic counit. We show that the inclusion ${\opCat{\calD_\bullet} \rightarrow (\complex/\Ring)_{f.g.}}$ satisfies the dual conditions.
Bear in mind that, by Lemma~\ref{LemJacobson}, the full subcategory  ${\opCat{\calD_\bullet} \rightarrow (\complex/\Ring)_{f.g.}}$ may be identified that of f.g. algebras that are reduced as rings.

The reflective subcategory ${\Red \rightarrow \Ring}$ of Lemma~\ref{LemReduced} induces another one such ${\complex/\Red \rightarrow \complex/\Ring}$. Also, the left adjoint 
${\complex/\Ring \rightarrow \complex/\Red}$ is again obtained by quotienting by the nilradical so the unit is again regular epic.
Moreover, it preserves monos because the canonical ${\complex/\Red \rightarrow \Red}$ reflects monos.

By Noetherianity, the nilradical of a finitely generated algebra is finitely generated so the left adjoint ${\complex/\Ring \rightarrow \complex/\Red}$ restricts to finitely generated algebras. That is, we have the reflective  ${\opCat{\calD_\bullet} \rightarrow (\complex/\Ring)_{f.g.}}$ satisfying the necessary conditions.
\end{proof}

Recall from Proposition~\ref{PropWeilSubquality}  that the Weil topos $\mathfrak{W}$ is a subtopos ${\mathfrak{W} \rightarrow \mathfrak{G}}$ of the Gaeta topos ${\mathfrak{G}}$ and that the subtopos is actually and essential subquality of the pre-cohesive ${\mathfrak{G} \rightarrow \Sets}$.

\begin{theorem}[The Weil subquality is level~$\epsilon$ of the Gaeta topos]\label{ThmLevelEpsilonOfGaeta}  
The essential subquality ${\mathfrak{W} \rightarrow \mathfrak{G}}$ is level~$\epsilon$ of the pre-cohesive ${p: \mathfrak{G} \rightarrow \Sets}$ and it coincides with the largest subquality of $p$.
\end{theorem}
\begin{proof}
We identify the Gaeta topos $\mathfrak{G}$ with ${\Psh{\calC}}$ where $\calC$ is the opposite of the category of finitely generated complex algebras with exactly two idempotents.
By Corollary~\ref{CorLargestSubqualityIsEssential} it is enough to prove that $\calC$ has enough little figures so let ${f : X \rightarrow Y}$ be a pseudo-constant in $\calC$.
Then $f$ is a pseudo-constant in $\calD$ and $X$ has a point because every object of $\calC$ has a point. By Lemma~\ref{LemPropIsApplicable},  $f$ factors (in $\calD$) via a object with exactly one point. This object is necessarily in $\calC$ so the factorization of $f$ is inside $\calC$.
\end{proof}

We see Theorem~\ref{ThmLevelEpsilonOfGaeta} as a confirmation of Lawvere's suggestion (quoted in the beginning of the paper) that ``the infinitesimal spaces may well negate the largest essential subtopos of the ambient one which" has the property that ``the four adjoint functors relating their topos to the base topos coalesce into two".

%

\section{The Weil subquality is level~$\epsilon$ of the Zariski topos of $\complex$}
\label{SecZariski}

We show that the level~$\epsilon$ of the Gaeta topos for $\complex$ factors through the Zariski topos and, as a level of the latter, it is level~$\epsilon$.
Some of the ideas involving restricted subqualities may be formulated at an elementary level. We deal with these first.
Let ${p : \calE \rightarrow \calS}$ be a local geometric morphism.

\begin{proposition}\label{PropRestrictedSubquality}
Let $j:\calE_j\to\calE$ and $k:\calE_k\to\calE$ be subtoposes of $\calE$ and assume that $k$ is above $j$. If
${\calE_j\to \calE}$ is above the centre of $p$ (so that ${\calE_k\to\calE}$ is also above the centre of $p$) then the following
hold:
\begin{enumerate}
\item The subtopos $\calE_j\to \calE_k$ is above the centre of ${p k:\calE_k\to\calS}$.
\item If $\calE_j\to \calE$ is a subquality of $p$, then the subtopos ${\calE_j\to\calE_k}$ is a subquality of ${p k :\calE_k\to\calS}$.
\item If $\calE_j\to\calE$ is essential, then so is ${\calE_j\to\calE_k}$.
\item If a subtopos $\calE_l\to\calE_k$ is above the centre of ${p k : \calE_k \to\calS}$, then the subtopos
${\calE_l\to\calE_k\to\calE}$ is above the centre of $p$.
\end{enumerate}
\end{proposition}
\begin{proof} 1.\ According to Corollary \ref{estaeralaproposicion3.1} (and with the same notation introduced there) 
the unit of  $\calE_j\to \calE_k$ is $(k^*\nu_{k_*})\cdot {\xi'}^{-1}$. When we apply $p_*k_*:\calE_k\to\calE$
we observe that $p_*k_*k^*\nu_{k_*}$ is an iso since $p_*k_*k^*\simeq p_*$, given that $k$ is above the centre of $p$,
and $p_*\nu$ is an iso, given that $j$ is above the centre of $p$.

2. This follows at once since $\calE_j\to\calS$ is a quality type regardless of whether we consider $\calE_j$ as a 
subtopos of $\calE$ or of $\calE_k$.

3. This follows at once form Remark \ref{remarka}.

4. We must show that $p_*$ inverts the unit of $\calE_l\to\calE_k\to\calE$ assuming that $p_*k_*$ inverts the unit of 
$\calE_l\to\calE_k$; but this follows at once since $p_*$ inverts the unit of $k:\calE_k\to\calE$ because $\calE_k$ is above the centre of $p$.
\end{proof}

Proposition~\ref{PropRestrictedSubquality} allows to show that, if level~$\epsilon$ is not just that but is also the largest subquality then we can restrict it to subtoposes that contain it.
More precisely:

\begin{corollary}\label{CorRestrictedLargestSubquality} 
Assume that ${p : \calE \rightarrow \calS}$ has a largest subquality ${\calE_{j} \rightarrow \calE}$ and that it is essential (so that ${\calE_j \rightarrow \calE}$ is level~$\epsilon$ of $p$). If ${k : \calE_k \rightarrow \calE}$ is a subtopos above $j$ then the subtopos ${\calE_{j} \rightarrow \calE_k}$ is the largest subquality of ${p k : \calE_k \rightarrow \calS}$ and it is essential (so that ${\calE_{j} \rightarrow \calE_k}$ is level~$\epsilon$ of  ${p k}$). 
\end{corollary}
\begin{proof}
By the second item of Proposition~\ref{PropRestrictedSubquality}, the subtopos ${\calE_j \rightarrow \calE_k}$ is a subquality of ${p k : \calE_k \rightarrow \calS}$ and it is essential by the third item.
Now assume that ${\calE_l \rightarrow \calE_k}$ is a subquality of ${p k : \calE_k \rightarrow \calS}$.
Then ${\calE_l \rightarrow \calE_k \rightarrow \calE}$ is a subquality of $p$ by the fourth item.
So it is above ${\calE_j \rightarrow \calE}$ by hypothesis and then, 
${\calE_l \rightarrow \calE_k}$ is above ${\calE_j \rightarrow \calE_k}$.
\end{proof}

We can now start to discuss the example.
It is convenient to give first an alternative presentation of the Gaeta topos of $\complex$ discussed in Section~\ref{SecWeilSubquality}.
As in that Section, let $\calD$ be the opposite of the category of finitely generated $\complex$-algebras.
Let ${J_G}$ be the Gaeta coverage on $\calD$.
The basic covering families are those of the form 
\[ ( D_i \rightarrow D  \mid i \in I ) \]
such that $I$ is finite and the induced ${\sum_{i\in I} D_i \rightarrow D}$ is an isomorphism.
The intimate relation between products in $\opCat{\calD}$ and  idempotents implies that the $J_G$-cocovering families in $\opCat{\calD}$ are those of  the form
\[ ( A \rightarrow A[a_i^{-1}] \mid i \in I ) \]
where $I$ is a finite set, ${\sum_{i\in I} a_i = 1}$ and, for every ${i, j \in I}$, ${i \not= j}$ implies ${a_i a_j = 0}$.

Let ${\calC \rightarrow \calD}$ be the full subcategory determined by the (f.g.) algebras that have exactly two idempotents. As every object of $\calD$ is a finite coproduct of objects in $\calC$,
the inclusion ${\calC \rightarrow \calD}$ is $J_G$-dense and so the Comparison Lemma (C2.2.3 in \cite{elephant}) implies that restricting along the inclusion ${\calC \rightarrow \calD}$ underlies an equivalence ${\Sh(\calD, J_G) \rightarrow \Psh{\calC} = \mathfrak{G}}$ between the topos of sheaves ${\Sh(\calD, J_G)}$ and the topos of presheaves ${\Psh{\calC}}$ that we used to define the Gaeta topos in Section~\ref{SecWeilSubquality}.

Let ${J_Z}$ be the Zariski coverage on $\calD$.
It is well-known that $J_Z$-cocovering basic families in $\opCat{\calD}$ are those of  the form
\[ ( A \rightarrow A[a_i^{-1}] \mid i \in I ) \]
where $I$ is a finite set and the ideal generated by ${(a_i \mid i \in I)}$ contains $1$. 
The topos ${\Sh(\calD, J_Z)}$ will be denoted by ${\mathfrak{Z}}$ and is called the {\em Zariski topos} (determined by the field $\complex$) and ${\mathfrak{Z} = \Sh(\calD, J_Z) \rightarrow \Set}$ is pre-cohesive (see \cite{Lawvere07} and also Example~{1.5} in \cite{Johnstone2011}).

The above description of $J_Z$ and $J_G$ implies that every $J_G$-cover is a $J_Z$ cover.
That is, the Zariski topos is a subtopos of the Gaeta topos.
This presentation of the Zariski topos as a subtopos 
$$\xymatrix{
\mathfrak{Z} = \Sh(\calD, J_Z) \ar[r] &  \Sh(\calD, J_G) \ar[r]_-{\cong} &  \Psh{\calC} = \mathfrak{G}
}$$
(of the Gaeta topos) whose direct image is restriction along ${\calC \rightarrow \calD}$ is motivated by the discussion starting at the end of p.~109 in \cite{Lawvere76}.

\begin{lemma}\label{LemAboveEpsilon} The subtopos ${\mathfrak{Z} \rightarrow  \mathfrak{G}}$ is above the Weil subquality ${\mathfrak{W} \rightarrow  \mathfrak{G}}$.
\end{lemma}
\begin{proof}
Recall from Section~\ref{SecWeilSubquality} that we identified the Weil subquality with the geometric inclusion induced by the full subcategory  ${\calC_!  \rightarrow \calC}$. 
Consider now the full inclusion ${\calC_! \rightarrow \calC \rightarrow \calD}$.
By Lemma~{C2.3.9} in \cite{elephant} there exists a smallest coverage $K$ on $\calC_!$ such that the inclusion into $\calD$ is cover reflecting.
In that result, $K$ is defined as the Grothendieck coverage generated by the sieves of the form
${R \cap \calC_{!}}$ where $R$ is ${J_Z}$-covering. We show below that all these sieves  contain an iso, which will allow us to conclude that $K$ is trivial.

Consider an object $A$  in $\opCat{\calC_!}$.
Since it has a unique maximal ideal,  Exercise 5.24 in \cite{AtiyahMacdonald} implies that the nilradical of $A$ is a maximal ideal. 
Thus, by Exercise 1.10 loc.~cit., every element of $A$ is either nilpotent or invertible. 
As $A$ is non-trivial, a Zariski cover cannot be generated by nilpotents, so every $J_Z$-cocover of $A$ contains an isomorphism. 
In other words, for every $C$ in $\calC_!$, the only $J_Z$-covering sieve of $C$ as an object of $\calD$ is the maximal one. This implies that $K$ is trivial.

The proof of  Lemma~{C2.3.9} cited above shows that the outer square below
$$\xymatrix{
\Sh(\calC_!, K) \ar[d] \ar[rr] &&  \ar@{.>}[ld] \Sh(\calD, J_Z) \ar[d] \\
\Psh{\calC_!} \ar[r] & \Psh{\calC} \ar[r] & \Psh{\calD}
}$$
is a pullback. As the right vertical map factors through ${\Psh{\calC} \rightarrow \Psh{\calD}}$,
the inner polygon is also a pullback and, since the left vertical map is an isomorphism,
${\Sh(\calD, J_Z) \rightarrow  \Psh{\calC}}$ is above ${\Psh{\calC_!} \rightarrow  \Psh{\calC}}$.
\end{proof}

We may now identify level~$\epsilon$ of the Zariski topos of $\mathbb{C}$.

\begin{theorem}[The Weil subquality is level~$\epsilon$ of the Zariski topos]\label{ThmZariski} 
The subtopos ${\mathfrak{W} \rightarrow \mathfrak{Z}}$ is level~$\epsilon$ of the pre-cohesive ${p : \mathfrak{Z} \rightarrow \Set}$ and it coincides with the largest subquality of $p$.
\end{theorem}
\begin{proof}
By Theorem~\ref{ThmLevelEpsilonOfGaeta}, the subtopos ${\mathfrak{W} \rightarrow \mathfrak{G}}$ is level~$\epsilon$ of the pre-cohesive ${p: \mathfrak{G} \rightarrow \Sets}$ and it coincides with the largest subquality of $p$.
Lemma~\ref{LemAboveEpsilon} shows ${\mathfrak{Z} \rightarrow\mathfrak{G}}$ is above ${\mathfrak{W} \rightarrow \mathfrak{G}}$ so Corollary~\ref{CorRestrictedLargestSubquality} applies.
\end{proof}

Corollary~\ref{CorRestrictedLargestSubquality} suggests the question: what is the largest subtopos of the Zariski topos that contains the Weil topos?
In any case, the calculation of the Aufhebung of $\epsilon$ still needs to be carried out.

\section*{Acknowledgements}
This paper slowly grew out of conversations with F.~W.~Lawvere during a seminar organized by F.~Marmolejo, which took place in April 2014 at Ciudad de M\'exico. We also thank the referee for several constructive suggestions. The second author would also like to thank the support of  Universit\`a di Bologna and funding from the European
Union's Horizon 2020 research and innovation programme under the Marie Sk\l odowska-Curie grant agreement No. 690974.


\end{document}